\documentclass[12pt,a4paper]{amsart}
\usepackage[utf8]{inputenc}
\usepackage[T1]{fontenc}
\usepackage{mathtools}
\usepackage{amsmath}
\usepackage{amsthm}
\usepackage{amssymb}
\usepackage[abbrev]{amsrefs}
\usepackage{mathrsfs}
\usepackage[dvipsnames]{xcolor}
\usepackage{bm}
\usepackage{enumitem}
\usepackage{hyperref}
\usepackage{graphicx}
\usepackage[all]{xy}
\usepackage{ifthen}
\AtBeginDocument{\def\MR#1{}}
\makeatletter
\@namedef{subjclassname@2020}{%
\textup{2020} Mathematics Subject Classification}
\makeatother
\numberwithin{equation}{section}
\allowdisplaybreaks
\newtheorem{Question}{}

\newtheorem{question}[Question]{Question}

\newtheorem{thm}{}[section]
\newtheorem{theorem}[thm]{Theorem}
\newtheorem{corollary}[thm]{Corollary}
\newtheorem{lemma}[thm]{Lemma}
\newtheorem{proposition}[thm]{Proposition}
\theoremstyle{definition}

\newcommand{\ee}{\ensuremath{\bm{e}}}%
\newcommand{\XX}{\ensuremath{\mathbb{X}}}%
\newcommand{\NN}{\ensuremath{\mathbb{N}}}%
\newcommand{\FF}{\ensuremath{\mathbb{F}}}%
\newcommand{\ZZ}{\ensuremath{\mathbb{Z}}}%
\newcommand{\YY}{\ensuremath{\mathbb{Y}}}%
\newcommand{\Ft}{\ensuremath{\mathcal{F}}}%
\newcommand{\Gt}{\ensuremath{\mathcal{G}}}%
\newcommand{\Pt}{\ensuremath{\mathcal{P}}}%
\newcommand{\St}{\ensuremath{\mathcal{S}}}%
\newcommand{\Sym}{\ensuremath{\mathbb{S}}}%
\newcommand{\Dt}{\ensuremath{\mathcal{D}}}%
\newcommand{\Jt}{\ensuremath{\mathcal{J}}}%
\newcommand{\At}{\ensuremath{\mathcal{A}}}%
\newcommand{\Bt}{\ensuremath{\mathcal{B}}}%
\newcommand{\YB}{\ensuremath{\mathcal{Y}}}%
\newcommand{\XB}{\ensuremath{\mathcal{X}}}%
\newcommand{\Et}{\ensuremath{\mathcal{E}}}%
\newcommand{\st}{\ensuremath{\bm{s}}}%
\newcommand{\yy}{\ensuremath{\bm{y}}}%
\newcommand{\xx}{\ensuremath{\bm{x}}}%
\newcommand{\zz}{\ensuremath{\bm{z}}}%
\newcommand{\unc}{\ensuremath{{\bm{k}}}}%
\newcommand{\uncr}{\ensuremath{{\widetilde{\bm{k}}}}}%
\newcommand{\ww}{\ensuremath{{\bm{w}}}}%
\newcommand{\primv}{\ensuremath{{\bm{\upsilon}}}}%
\newcommand{\primp}{\ensuremath{{\bm{\pi}}}}%
\newcommand{\prim}{\ensuremath{{\bm{\sigma}}}}%
\newcommand{\primt}{\ensuremath{{\bm{\tau}}}}%
\DeclareMathOperator{\supp}{supp}%
\DeclareMathOperator{\spn}{spn}%
\newcommand{\abs}[1]{\left\lvert#1\right\rvert}
\newcommand{\norm}[1]{\left\lVert#1\right\rVert}
\newcommand{\enfloor}[1]{\left\lfloor#1\right\rfloor}

\newcommand{\enbrace}[1]{\left\lbrace#1\right\rbrace}

\newcommand{\enpar}[1]{\left(#1\right)}
\newcommand{\Env}[2][]{\ifthenelse{ \equal{#1}{}}{\ensuremath{#2_{\mathsf{c}}}}{\ensuremath{#2_{\mathsf{c},#1}}}}
\author[F. Albiac]{Fernando Albiac}
\address{Fernando Albiac\\
Department of Mathematics, Statistics, and Computer Sciences--Ina\-Mat$^2$ \\
Universidad P\'ublica de Navarra\\
Campus de Arrosad\'{i}a\\
Pamplona\\
31006 Spain}
\email{fernando.albiac@unavarra.es}
\author[J. L. Ansorena]{Jos\'e L. Ansorena}
\address{Jos\'e L. Ansorena\\
Department of Mathematics and Computer Sciences\\
Universidad de La Rioja\\
Logro\~no\\
26004 Spain}
\email{joseluis.ansorena@unirioja.es}
\author[M. Berasategui]{Miguel Berasategui}
\address{Miguel Berasategui\\
UBA - Pab I, Facultad de Ciencias Exactas y Naturales\\
Universidad de Buenos Aires\\
(1428), Buenos Aires, Argentina}
\email{mberasategui@dm.uba.ar}
\subjclass[2020]{46A16, 46A35, 46A45, 46B15, 46B45}
\keywords{Quasi-Banach spaces, averaging in Banach spaces, symmetric bases, Lorentz spaces, conditional bases, almost greedy bases, order estimates, Banach lattices, unconditonality constants}
\begin{document}
%--------------------------------------------
\title[Averaging projections in quasi-Banach spaces]{Boundedness of the averaging projections in nonlocally convex Lorentz sequence spaces and applications to basis theory}
%--------------------------------------------
\begin{abstract}
We study the boundedness of averaging projections associated with symmetric Schauder bases in quasi-Banach spaces. Although this property is standard in the Banach setting, it is far from clear in the absence of local convexity and, indeed, fails for a broad class of quasi-Banach spaces with a symmetric basis, including $\ell_p$ for $0<p<1$. Our main result shows that, nevertheless, the canonical basis of an entire class of weighted Lorentz sequence spaces, including the spaces $\ell_{p,q}$ for $0<q<1<p<\infty$, has uniformly bounded averaging projections. Thus bounded averaging projections do not characterize local convexity among quasi-Banach spaces with symmetric bases. As applications, we obtain new consequences for the structure of special bases. In particular, as a byproduct of our approach, we derive new examples of conditional and almost greedy bases in nonlocally convex spaces.
\end{abstract}
%--------------------------------------------
\thanks{F.\@ Albiac and J.\@ L.\@ Ansorena acknowledge the support of the Spanish Ministry for Science and Innovation under Grant PID2022-138342NB-I00 for \emph{Functional Analysis Techniques in Approximation Theory and Applications}}
%--------------------------------------------------------------------------
\maketitle
%--------------------------------------------
\section{Introduction}\noindent
%--------------------------------------------
Averaging operators are among the most natural transformations associated with a basis. Recall that given a Schauder basis $\XB=(\xx_n)_{n=1}^\infty$ of a Banach (or more generally, a quasi-Banach) space $\XX$, and a family $\At=(A_j)_{j\in J}$ of pairwise disjoint nonempty finite subsets of $\NN$, the \emph{averaging projection on $\XX$ relative to $\XB$ and $\At$} is the linear map that assigns to each $f=\sum_{n=1}^\infty a_n \xx_n\in \XX$ the vector
\[
V_{\At}[\XB](f)
=\sum_{j\in J} \frac{1}{\abs{A_j}} \enpar{\sum_{n\in A_j} a_n} \enpar{\sum_{n\in A_j} \xx_n}\in\XX.
\]
When the basis $\XB$ is symmetric, the boundedness of these operators expresses a strong form of stability under blockwise averaging. In Banach spaces, averaging projections occur naturally in the geometry of symmetric basic sequences, in the study of complemented block subspaces, and in the construction of special kinds of conditional bases. However, the standard argument that yields the boundedness of averaging projections relative to symmetric Schauder bases (see, e.g., \cite{AlbiacKalton2016}*{Theorem 9.2.6}) relies heavily on local convexity. Thus, once one passes from Banach spaces to quasi-Banach spaces, it is no longer clear whether boundedness of averaging projections should remain available.

This is not merely a technical issue. Indeed, in nonlocally convex spaces, averaging need not behave as a regularizing operation. On the contrary, quasi-norms often react badly to the spreading of mass, so one cannot expect blockwise averaging to be automatically bounded. This leads to the following natural question.

\begin{question}\label{qt:LNCAP}
Is there a nonlocally convex quasi-Banach space with a symmetric Schauder basis relative to which all averaging projections are bounded?
\end{question}

There are solid reasons to suspect that the compatibility of the quasi-norm with the boundedness of averaging projections might force local convexity. Kalton \cite{Kalton1984ArchMath} proved that for vector-valued $L_p(\XX)$, $0<p<\infty$, the existence of an averaging projection onto $\XX$ implies local convexity of $\XX$ under mild additional assumptions, and he conjectured that if such averaging projection exists, then $1\le p<\infty$ and $\XX$ is locally convex (see \cite{Kalton1984ArchMath}*{p.\@ 73}). Even though that problem lives in a function-space setting, it helped shape the intuition that bounded averaging is too rigid to occur in a genuinely nonlocally convex environment.

Altogether, these facts suggest that bounded averaging relative to bases should be a rather restrictive feature outside the locally convex setting. The purpose of this paper is to show that this intuition is false.

Hence, the boundedness of the averaging projections does not characterize local convexity and it should be viewed instead as an intrinsic geometric property of rearrangement invariant quasi-Banach spaces in its own right, rather than as a proxy for local convexity. The present work provides evidence that, in the quasi-Banach setting, the geometry of the space can compensate for the lack of convexity in ways that had not been previously recognized. The result is especially striking in view of the obstructions that we will exhibit in detail in Section~\ref{prelim}. In particular, if $\XB$ is a symmetric Schauder basis and all averaging projections relative to $\XB$ are bounded, then the fundamental function of the basis behaves like the fundamental function of a basis of a Banach space.

Lorentz sequence spaces $\ell_{p,q}$, $0<p<\infty$, $0<q\le \infty$, provide a natural testing ground for Question~\ref{qt:LNCAP}. Their canonical basis is symmetric, and its fundamental function grows like the sequence
\[
\primp_p=(m^{1/p})_{m=1}^\infty,
\]
exactly as in the classical Banach space $\ell_p$ when $p\ge 1$. Thus, the spaces $\ell_{p,q}$ for $p\ge 1$ are plausible candidates for carrying bounded averaging projections even in the absence of local convexity. At the same time, other obstructions exclude the parameter ranges $p=1$ and $q\not=1$. If $p>1$ and $q>1$, then $\ell_{p,q}$ is locally convex. Thus the only genuinely non-Banach range left open is $0<q<1<p<\infty$. As we will see, this is precisely the non-locally convex regime in which averaging projections are bounded. In fact, we prove a more general boundedness theorem for weighted Lorentz spaces $d_q(\prim)$ associated with an increasing sequence $\prim$ and an index $0<q<1$. The key point is that suitable regularity assumptions on $\prim$ allow one to control the Lorentz norm of the averages despite the non-locally convex nature of the $\ell_q$-norm. The result for classical Lorentz spaces $\ell_{p,q}$ follows by applying this general result to the sequence $\primp_p$.

Beyond its intrinsic interest, the boundedness of averaging projections has applications to the structure of bases in quasi-Banach spaces. In the Banach space setting, a classical theorem of Pe{\l}czy\'nski and Singer \cite{PelSin1964} asserts that every Banach space with a Schauder basis admits a continuum of mutually non-equivalent conditional Schauder bases. However, this result gives little information about the geometry of such bases. In order to construct conditional bases with additional structure, one typically needs extra geometric information on the ambient space. In this direction, the existence of a complemented subspace with a symmetric Schauder basis is especially useful. For instance, the Dilworth--Kalton--Kutzarova method invented in \cite{DKK2003} can be exploited to construct almost greedy bases with prescribed growth of their unconditionality parameters (see \cite{AADK2019b}) in some Banach spaces.

In the nonlocally convex setting, even the validity of the Pe{\l}czy\'nski--Singer theorem is unclear, since its proof depends on the Hahn--Banach theorem. Thanks to the boundedness of the averaging projections and a further extension of the Dilworth--Kalton--Kutzarova method, the mechanism based on complemented symmetric subspaces remains available in certain non-locally convex spaces, and this leads to new examples of special types of conditional bases with prescribed growth of their unconditionality parameters in that setting.

The paper is organized as follows. In Section~\ref{prelim}, we collect the preliminary results on quasi-Banach spaces, bases, atomic lattices, and fundamental functions that will be used throughout the paper. In Section~\ref{sect:RP}, we study weighted Lorentz sequence spaces and prove the boundedness of averaging projections under suitable regularity assumptions on the weight. In Section~\ref{sect:TGA}, we turn to applications to the structure of bases, in particular to complemented block sequences and to the construction of conditional and almost greedy bases in nonlocally convex settings.

Throughout this article we will use the standard notation and language in modern Banach space theory and in particular in the theory of bases, as can be found in \cite{AlbiacKalton2016}. However, for the convenience of the reader we will next recall the most heavily used terminology.
%--------------------------------
\section{Preliminaries and structural obstructions}\label{prelim}\noindent
%--------------------------------
In this section we gather the notation and background material needed later. We review basic terminology on quasi-Banach spaces and their bases, as well as a few structural facts concerning symmetric bases with bounded averaging projections.
%--------------------------------
\subsection{Quasi-Banach spaces, bases, and lattices}
%--------------------------------
Let $0<q\le 1$. A topological vector space is said to be \emph{locally $q$-convex} if it admits a neighborhood basis of the origin consisting of $q$-convex sets. A quasi-norm on a vector space $\XX$ over the (real or complex) field $\FF$ is a map
\[
\norm{\cdot}\colon \XX\to [0,\infty)
\]
such that $\norm{f}>0$ unless $f=0$, $\norm{\lambda f}=\abs{\lambda}\norm{f}$ for all $\lambda\in\FF$ and all $f\in \XX$, and
\[
\norm{f+g}\le \kappa\enpar{\norm{f}+\norm{g}}, \qquad f,g\in \XX,
\]
for some constant $\kappa$. If, more strongly,
\[
\norm{f+g}^q\le \norm{f}^q+\norm{g}^q, \qquad f,g\in \XX,
\]
then $\norm{\cdot}$ is called a $q$-norm. By the Aoki--Rolewicz theorem \cites{Aoki1942,Rolewicz1957}, every locally bounded topological vector space is locally $q$-convex for some $q\le 1$, whence it can be equipped with a $q$-norm. A quasi-Banach space (resp., $q$-Banach or Banach) is a complete topological vector space equipped with a quasi-norm (resp., $q$-norm or norm).

A subspace $\YY$ of a quasi-Banach space $\XX$ is said to be \emph{complemented} if there exists a bounded projection from $\XX$ onto $\YY$.

Given two quasi-Banach spaces $\XX$ and $\YY$, we write $\XX\oplus \YY$ for their direct sum. Given Schauder bases $\XB=(\xx_n)_{n=1}^\infty$ and $\YB=(\yy_n)_{n=1}^\infty$ of quasi-Banach spaces $\XX$ and $\YY$, respectively, $\XB\oplus\YB=(\zz_n)_{n=1}^\infty$ will be the Schauder basis of $\XX\oplus\YY$ defined by
\[
\zz_{2n-1}=(\xx_n,0), \quad \zz_{2n}=(0,\yy_n), \quad n\in\NN.
\]

A family $(f_j)_{j\in J}$ in a quasi-Banach space $\XX$ is said to be \emph{semi-normalized} if
\[
0<\inf_{j\in J}\norm{f_j}\le \sup_{j\in J}\norm{f_j}<\infty.
\]

Let $\XB=(\xx_n)_{n=1}^\infty$ be a sequence in a quasi-Banach space $\XX$. We say that $\XB$ is a \emph{Schauder basis} if every $f\in \XX$ admits a unique expansion
\[
f=\sum_{n=1}^\infty a_n \xx_n.
\]
If the series converge unconditionally, we say that $\XB$ is an \emph{unconditional basis}. The \emph{dual basis} $\XB^*=(\xx_n^*)_{n=1}^\infty$ of a Schauder basis $\XB=(\xx_n)_{n=1}^\infty$ of a quasi-Banach space $\XX$ is the sequence in $\XX^*$ defined for each $k\in\NN$ by
\[
\xx_k^*\colon\XX\to \FF, \quad \sum_{n=1}^\infty a_n\, \xx_n \mapsto a_k.
\]
A \emph{block basic sequence} relative to $\XB$ is a sequence $(f_k)_{k=1}^\infty$ such that there are scalars $(a_n)_{n=1}^\infty$ and integers
\[
0=m_0<m_1<m_2<\cdots
\]
with
\[
f_k=\sum_{n=m_{k-1}+1}^{m_k} a_n \, \xx_n \not=0,\qquad k\in\NN.
\]

Set $c_{00}=c_{00}(\NN)$. Given $\lambda=(\lambda_n)_{n=1}^\infty\in c_{00}$ and a Schauder basis $\XB=(\xx_n)_{n=1}^\infty$ of $\XX$, we define the corresponding \emph{multiplier operator}
\[
M_\lambda[\XB]\colon \XX\to \XX,\qquad \sum_{n=1}^\infty a_n\xx_n\mapsto \sum_{n=1}^\infty \lambda_n a_n\xx_n.
\]
If $\lambda=\chi_A$ is the indicator function of a finite set $A\subset \NN$, we write
\[
S_A[\XB]=M_{\chi_A}[\XB].
\]
As in the locally convex setting, the following are equivalent:
\begin{itemize}
\item $\XB$ is unconditional;
\item the family $M_\lambda[\XB]$, $\lambda\in c_{00}\cap B_{\ell_\infty}$, is uniformly bounded;
\item the family $S_A[\XB]$, $A\subset \NN$, $\abs{A}<\infty$, is uniformly bounded.
\end{itemize}

Let $\Pi(\NN)$ denote the set of permutations of $\NN$. Given $\pi\in \Pi(\NN)$, we denote by
\[
R_\pi[\XB]:\operatorname{span}(\XB)\to \operatorname{span}(\XB)
\]
the linear map defined by $\xx_n\mapsto \xx_{\pi(n)}$. A Schauder basis $\XB$ is said to be \emph{symmetric} if it is equivalent to all its permutations. We say that $\XB$ is \emph{subsymmetric} if it is unconditional and equivalent to all its subsequences.
Equivalently, $\XB$ is symmetric if and only if the family
\[
R_\pi[\XB]\circ M_\lambda[\XB], \qquad \lambda\in c_{00}\cap B_{\ell_\infty},\ \pi\in\Pi(\NN),
\]
is uniformly bounded. Any symmetric basis is subsymmetric.

The support of $f\in \XX$ relative to $\XB$ is the set
\[
\operatorname{supp}(f)=\{n\in\NN\colon \xx_n^*(f)\ne 0\},
\]
where $\XB^*=(x_n^*)_{n=1}^\infty$ denotes the dual basis.

The unconditionality parameters of $\XB$ are defined by
\[
\unc_m[\XB]=\sup_{\abs{A}\le m}\norm{S_A[\XB]},\qquad m\in\NN.
\]
We also consider the localized parameters
\[
\uncr_m[\XB]
=\sup\enbrace{\norm{S_A[\XB](f)} \colon \norm{f}=1,\ \supp(f)\subset [1,m],\ A\subset \NN}.
\]
Then $\uncr_m[\XB]\le \unc_m[\XB]$ for all $m$. Still, $\XB$ is unconditional if and only if $\sup_m \uncr_m[\XB]<\infty$.

A finite set $A\subset \NN$ is called a \emph{greedy set} of $f=\sum_{n=1}^\infty a_n\xx_n$ relative to $\XB$ if
\[
\abs{a_n}\ge \abs{a_k}, \qquad n\in A,\ k\in \NN\setminus A.
\]
If $\XB$ is semi-normalized, we say that it is \emph{almost greedy} if there exists $C\ge 1$ such that
\[
\norm{f-S_A[\XB](f)}
\le C \norm{f-S_B[\XB](f)}
\]
for every $f\in \XX$, every greedy set $A$ of $f$, and every $B\subset \NN$ with $\abs{B}\le \abs{A}$.

An \emph{atomic quasi-Banach lattice} is a quasi-Banach space $\XX$ such that $c_{00}\subset \XX\subset \FF^{\NN}$ and $\norm{g}\le\norm{f}$ for all $f$, $g\in\XX$ with $\abs{g}\le \abs{f}$. If the unit vector system spans a dense subspace of $\XX$, we say that $\XX$ is \emph{minimal}. If
\[
\norm{(f_{\pi(n)})_{n=1}^\infty}=\norm{f}
\]
for all $f\in \XX$ and $\pi\in\Pi(\NN)$, then $\XX$ is said to be \emph{rearrangement invariant}.

The usual correspondence between unconditional bases and atomic quasi-Banach lattices (regarded as sequence spaces with the canonical basis $(\ee_n)_{n=1}^\infty$) allows us to transfer lattice properties to bases, and conversely. Via this correspondence, symmetric bases induce rearrangement invariant quasi-Banach lattice structures.

Let $\XX$ be an atomic quasi-Banach lattice or, more generally, a quasi-Banach lattice. Given $0<r<\infty$ and a finite family $\varphi=(f_j)_{j\in J}$ in $\XX$, define
\[
M_r(\varphi)=\enpar{\norm{\sum_{j\in J}\abs{f_j}^r}^{1/r}},
\qquad
N_r(\varphi)=\enpar{\sum_{j\in J}\norm{f_j}^r}^{1/r}.
\]
We say that $\XX$ is \emph{lattice $r$-convex} if
\[
M_r(\varphi)\le C\,N_r(\varphi),
\]
and \emph{lattice $r$-concave} if
\[
N_r(\varphi)\le C\,M_r(\varphi)
\]
for every finite family $\varphi$. If these inequalities are required to hold only for disjointly supported families, then one speaks of an \emph{upper $r$-estimate} and a \emph{lower $r$-estimate}, respectively. If $\XX$ is a $r$-convex quasi-Banach lattice, then $\XX$ is a $\min\{r,1\}$-convex quasi-Banach space. The converse does not hold; indeed there are quasi-Banach lattices that are not $r$-convex for any $r>0$ (see \cite{Kalton1984b}). We say that $\XX$ if $L$-convex (resp., $L$-concave) if it is lattice $r$-convex (resp., $r$-concave) for some $0<r<\infty$. A quasi-Banach lattice is $L$-concave if and only if is $L$-convex and satisfies a lower $r$-estimate \cite{Kalton1984b}.

Finally, given $0<p<\infty$, the $p$-convexification of $\XX$ is the quasi-Banach lattice
\[
\XX^{(p)}=\{f\colon \abs{f}^p\in \XX\}.
\]
The lattice estimates defined above plainly pass from $\XX$ to $\XX^{(p)}$ by replacing the index $r$ with $pr$.

Given families $\lambda=(\lambda_j)_{j\in J}$ and $\mu=(\mu_j)_{j\in J}$ in $[0,\infty]$, we say that $\mu$ dominates $\lambda$, and we put $\lambda_j\lesssim \mu_j$ for $j\in J$, or $\lambda\lesssim\mu$, if there is a constant $C$ such that $\lambda_j\le C\mu_j$ for all $j\in J$. If $\lambda$ dominates $\mu$ and $\mu$ dominates $\lambda$, we say that $\lambda$ and $\mu$ are equivalent, and we put $\lambda_j\approx \mu_j$ for $j\in J$, or simply $\lambda\approx\mu$.

Let $J$ be a non-empty set, $\Ft=(f_j)_{j\in J}$ a family in $\XX$, and $C(\Ft)$ the family consisting of the quasi-norms of all linear combinations of vectors in $\Ft$, that is,
\[
C(\Ft)= \enpar{ \norm{\sum_{j\in J} a_j \, f_j}}_{(a_j)_{j\in J} \in c_{00}(J)},
\]
where $c_{00}(J)$ stands for the set of all vectors in $\FF^J$ with at most a finite number of nonzero coordinates. Let $\Gt=(g_j)_{j\in J}$ be another family in $\XX$ modeled over the same set $J$. We say that $\Gt$ \emph{dominates} (resp., \emph{is equivalent to}) $\Ft$ if $C(\Gt)$ dominates (resp., is equivalent to) to $C(\Ft)$. In other words, $\Gt$ dominates (resp., is equivalent to) $\Ft$ if and only if there is a linear continuous map (resp., an isomorphism)
\[
T\colon \spn(\Ft) \to \spn(\Gt)
\]
such that $T(f_j)=g_j$ for all $j\in J$.
%-----------------------------
\subsection{Fundamental functions, Banach envelopes, and averaging projections}
%-----------------------------
Let $\XB=(\xx_n)_{n=1}^\infty$ be a Schauder basis of a quasi-Banach space $\XX$. Its \emph{fundamental function} is defined by
\[
\Gamma_m[\XB]=\sup_{\substack{\abs{A}\le m\\\abs{\varepsilon_n}=1}}\norm{\sum_{n\in A}\varepsilon_n \xx_n}, \quad m\in\NN.
\]
If $\XX$ is an atomic quasi-Banach lattice, we will denote by $(\Gamma_m[\XX])_{m=1}^\infty$ the fundamental function of its canonical basis. If the basis $\XB$ is symmetric, then it is super-democratic, in the sense that
\[
\norm{\sum_{n\in A}\varepsilon_n \xx_n}\approx \Gamma_{\abs{A}}[\XB],
\]
where $(\varepsilon_n)_{n\in A}$ runs over all families of unimodular scalars modeled over a finite set $A\subset \NN$.

The \emph{Banach envelope} of a quasi-Banach space $\XX$ is a pair $(J_\XX,\widehat{\XX})$, where $\widehat{\XX}$ is a Banach space and $J_\XX\colon\XX\to \widehat{\XX}$ is a linear contraction, satisfying the following universal property: for every Banach space $\YY$ and every linear contraction $T\colon\XX\to \YY$ there is a linear contraction $\widehat{T}\colon\XX\to \YY$ such that $\widehat{T}\circ J_\XX=T$. If $\XB=(\xx_n)_{n=1}^\infty$ is a Schauder basis of $\XB$, then
\[
\widehat{\XB}:=\enpar{J_\XX(\xx_n)}_{n=1}^\infty
\]
is a Schauder basis of $\widehat{\XX}$. We call $\widehat{\XB}$ the Banach envelope of $\XB$. Both the Banach envelope basis and the dual basis of a symmetric basis are symmetric.

The boundedness of the averaging projections on a quasi-Banach space relative to a basis imposes some restrictions on the fundamental function of the basis.

\begin{lemma}\label{lem:fundamental-envelope-dual}
Let $\XB$ be a symmetric Schauder basis of a quasi-Banach space $\XX$. Assume that all the averaging projections relative to $\XB$ are bounded. Then
\[
\Gamma_m[\XB]
\approx\Gamma_m[\widehat{\XB}]
\approx\frac{m}{\Gamma_m[\XB^*]}, \quad m\in\NN.
\]
\end{lemma}

\begin{proof}
Pick a partition $\At=(A_j)_{j\in J}$ of $\NN$ such that for each $m\in\NN$ there is $j\in J$ with $\abs{A_j}=m$. Since $\XB$ is unconditional and $V_{\At}[\XB]$ is bounded, $\Gamma_m[\XB] \Gamma_m[\XB^*] \lesssim m$ for all $m\in\NN$. Since $m\le \Gamma_m[\widehat{\XB}] \Gamma_m[\XB^*]$ and $\Gamma_m[\widehat{\XB}]\le \Gamma_m[\XB]$ for all $m\in\NN$, we are done.
\end{proof}

As an immediate consequence of Lemma~\ref{lem:fundamental-envelope-dual}, one gets the following restriction.

\begin{lemma}\label{lem:gamma-linear}
Let $\XB$ be a symmetric Schauder basis of a quasi-Banach space $\XX$. Assume that all the averaging projections relative to $\XB$ are bounded. Then:
\begin{enumerate}[label=(\roman*),leftmargin=*,widest=ii]
\item\label{it:lem:gamma-lineal:a} $\Gamma_m[\XB]\lesssim m$ for all $m\in\NN$;
\item\label{it:lem:gamma-lineal:b} If $\Gamma_m[\XB]\approx m$ for all $m\in\NN$, then $\XB$ dominates the canonical basis of $\ell_1$.
\end{enumerate}
\end{lemma}

Lemma~\ref{lem:gamma-linear} shows that the boundedness of the averaging projections is a strong geometric restriction. Another obstruction comes from the uniqueness of the lattice structure induced by an unconditional basis. Recall that a quasi-Banach space $\XX$ with an unconditional basis $\XB$, which we assume to be normalized, is said to have a \emph{unique unconditional basis} if any other normalized unconditional basis of $\XX$ is equivalent to a permutation of $\XB$.

If the Schauder basis $\XB$ is symmetric, in particular it is unconditional and equivalent to all its permutations. Lindenstrauss and Zippin proved that if a Banach space with a symmetric basis has a unique unconditional basis then it must be equivalent to one of the three spaces $c_{0}$, $\ell_{1}$, or $\ell_{2}$. Now we can bring to play the boundedness of the averaging projections to extend this result to general quasi-Banach spaces.

\begin{theorem}[cf.\@ \cite{LinZip1969}*{Theorem 1}]\label{thm:LZ}
Let $\XX$ be a quasi-Banach space with a symmetric basis $\XB$. Assume that $\XX$ has a unique unconditional basis and that the averaging projections relative to $\XB$ are bounded. Then $\XX$ is isomorphic to $\ell_1$, $\ell_2$, or $c_0$.
\end{theorem}

\begin{corollary}
The averaging projections relative to the canonical bases in the spaces $\ell_{1,q}$ for $0<q<1$, and $\ell_{p,q}$ for $0<p<1$, $0<q\le 1$, $p\ne q$, are not bounded.
\end{corollary}

\begin{proof}
Notice that if $0<r<1<s$, then $\ell_{1,r}\subsetneq \ell_1\subsetneq\ell_{1,s}$, and so $\ell_{1,r}$ has a unique unconditional basis \cite{AlbiacLeranoz2008}. Hence, Lemma~\ref{lem:gamma-linear}\ref{it:lem:gamma-lineal:b} and Theorem~\ref{thm:LZ} rule out the possibility that the averaging projections relative to the canonical basis of $\ell_{1,q}$, $q\not=1$, are bounded. If $p<1$, the boundedness of averaging projections relative to the canonical basis of $\ell_{p,q}$ is ruled out using Lemma~\ref{lem:gamma-linear}\ref{it:lem:gamma-lineal:a} or Theorem~\ref{thm:LZ}.
\end{proof}
%--------------------------------------------
\section{Averaging projections in weighted Lorentz sequence spaces}\label{sect:RP}\noindent
% -------------------------------------------
In this section we prove the main result of the paper. As we announced in the Introduction,
the boundedness of averaging projections in the Lorentz spaces $\ell_{p,q}$, $0<q<1<p<\infty$, will be obtained as a consequence of a more general theorem for weighted Lorentz sequence spaces. There are two reasons for working in this greater generality. First, the proof itself depends much more on the regularity properties of the underlying weight than on the special form of the sequence $\primp_p$. Second, phrasing the argument in terms of weights clarifies which part of the mechanism is specific to classical Lorentz spaces and which part only uses abstract rearrangement estimates.

Weighted Lorentz sequence spaces provide a flexible scale of rearrangement invariant atomic quasi-Banach lattices whose geometry is governed by two parameters: an index $q\in(0,\infty]$ that measures how the size of the coordinates is aggregated, and a nondecreasing weight $\prim=(s_m)_{m=1}^\infty$ that determines the growth of the fundamental function. For our purposes, the advantage of the weighted setting is that it allows us to isolate the precise hypotheses on the weight that guarantee boundedness of the averaging projections.

Let $c_{00}^+$ be the positive cone of $c_{00}$. We denote by $D(g)$ the \emph{nonincreasing} rearrangement of $g\in c_{00}^+$. Given $g\in c_{00}$ we set
\[
\norm{g}_{q,\prim}=\enpar{\sum_{n=1}^\infty (s_n a_n)^q \frac{s_n-s_{n-1}}{s_n}}^{1/q}, \, D(g)=(a_n)_{n=1}^\infty,
\]
with the convention that $s_0=0$ and the standard modification if $q=\infty$. It is known \cite{AABW2021}*{Equation (9.3)} that, for all $0<p<\infty$,
\begin{equation}\label{eq:Lorentzpqw}
\norm{g}_{q,\prim}\approx \enpar{\sum_{n=1}^\infty (s_n a_n)^q \frac{s_n^p-s^p_{n-1}}{s_n^p}}^{1/q}, \, g\in c_{00}^+, \, D(g)=(a_n)_{n=1}^\infty.
\end{equation}
Now, we extend $\norm{\cdot}_{q,\prim}$ to a function from $\FF^\NN$ to $[0,\infty]$ by
\[
\norm{f}_{q,\prim}=\sup \enbrace{ \norm{g}_{q,\prim} \colon g\in c_{00}^+, \, g\le\abs{f}}.
\]
The Lorentz space $d_q(\prim)$ consists of all $f\in\FF^\NN$ such that $\norm{f}_{q,\prim}<\infty$. If the sequences $\prim$ and $\primt$ satisfy $\prim\lesssim \primt$, then $\norm{f}_{q,\prim}\lesssim \norm{f}_{q,\primt}$ for all $f\in c_{00}^+$. In particular, if $\prim\approx\primt$ then $d_q(\prim)=d_q(\primt)$. These connections between different gauges and spaces follow from the Abel summation formula, which states that
\begin{equation*}
\sum_{n=1}^\infty a_n b_n=\sum_{n=1}^\infty \enpar{a_n-a_{n+1}}\sum_{j=1}^n b_j
\end{equation*}
for all eventually null sequences $(a_n)_{n=1}^\infty$ and $(b_n)_{n=1}^\infty$.

If $\prim$ is \emph{doubling}, that is, $s_{2m}\lesssim s_m$ for $m\in\NN$, then $d_q(\prim)$ is a rearrangement invariant atomic quasi-Banach lattice. Further, $d_q(\prim)$ is minimal if and only if $q<\infty$ and $\prim$ is unbounded. In the trivial case that $\prim$ is bounded, then $d_q(\ww)=\ell_\infty$. By \eqref{eq:Lorentzpqw},
\begin{equation}\label{eq:FFL}
\Gamma_m[d_q(\prim)] \approx \prim.
\end{equation}
In a sense, the sequence $\prim$ places $d_q(\prim)$ within a scale of spaces and the index $q\in(0,\infty]$ determines its position in this scale. In fact,
\[
d_q(\prim) \subset d_r(\prim), \quad 0<q<r\le\infty.
\]
Given $0<p<\infty$,
\begin{equation}\label{eq:LatticeConvex}
\enpar{d_q(\prim)}^{(p)}=d_{pq}(\prim^{1/p}).
\end{equation}

Hunt \cite{Hunt1966} carried out a systematic study of classical function and sequence spaces $L_{p,q}$ for $0<p<\infty$ and $0<q\le\infty$. The study of the geometry of Lorentz sequence spaces was pioneered in \cite{ACL1973}. There, the authors picked $0<q<\infty$ and a nonnegative sequence $\ww=(w_n)_{n=1}^\infty$ with $w_1>0$, and defined the weighted Lorentz space $d(\ww,q)$ as the completion of the space of all $f\in c_{00}$ such that
\[
\norm{f}_{(\ww,q)}= \enpar{\sum_{n=1}^\infty a_n^q w_n}^{1/q}, \quad (a_n)_{n=1}^\infty=D(\abs{f}).
\]
Using our notation, equation\eqref{eq:Lorentzpqw} yields $d(\ww,q)=d_q(\prim)$, where $\prim=(s_m)_{m=1}^\infty$ is the sequence
\begin{equation}\label{eq:ws}
s_m=\enpar{\sum_{n=1}^m w_n}^{1/q}.
\end{equation}
In particular, if $0<p<\infty$ we obtain $\ell_{p,q}=d_q(\primp_p)$. If $\ww$ is decreasing, then $d(\ww,q)$ is a $p$-Banach space for $p=\min\{1,q\}$.

Altshuler et al.\@ investigated the spaces $d(\ww,q)$ in the case when $\ww$ is nonincreasing and $q\ge 1$, so that $d(\ww,q)$ is a Banach space. This premise was maintained in several papers that followed (see \cites{ACL1973,Altshuler1975,Allen1978}). Popa \cite{Popa1981} studied the nonlocally Lorentz spaces $d(\ww,q)$ for $q<1$ and $\ww$ nonincreasing. He investigated their Mackey topology and their subspace structure, and proved the following result.

\begin{theorem}[\cite{Popa1981}*{Lemma 3.1}]\label{thm:Popa}
Let $0<q<1$ and $\ww=(w_n)_{n=1}^\infty$ be a nonincreasing nonsummable weight. Assume that the sequence $(s_m)_{m=1}^\infty$ defined in\eqref{eq:ws} is unbounded. Let $(f_k)_{k=1}^\infty$ be a semi-normalized block basic sequence of $d(\ww,q)$. If $\lim_k \norm{f_k}_\infty=0$, then a subsequence of $(f_k)_{k=1}^\infty$ is equivalent to the canonical $\ell_q$-basis.
\end{theorem}

Nawrocki and Orty\'{n}ski \cite{NawOrt1985} solved several questions raised in \cite{Popa1981}. One of their advances was finding instances where $\ell_q$ is a complemented subspace of $d(\ww,q)$.

\begin{theorem}[\cite{NawOrt1985}*{Theorem 2}]\label{thm:NO}
Let $0<q<1$ and $\ww=(w_n)_{n=1}^\infty$ be a nonincreasing nonsummable weight. Let $(s_m)_{m=1}^\infty$ be as in\eqref{eq:ws}. Assume that $\lim_m s_m/m=0$. Then there is a block basic sequence of $d(\ww,q)$ that is complemented and equivalent to the canonical $\ell_q$-basis.
\end{theorem}

Twenty-five years later, Albiac and Leranoz \cite{AlbiacLeranoz2008} proved the following result that rules out the possibility that $\ell_q$ is isomorphic to a complemented subspace of $d(\ww,q)$ in the opposite case that $\inf_m s_m/m>0$.

\begin{theorem}[\cite{AlbiacLeranoz2008}*{Theorem 2.4}]
Let $0<q<1$ and $\ww=(w_n)_{n=1}^\infty$ be a nonincreasing nonsummable weight. Let $(s_m)_{m=1}^\infty$ be defined as in\eqref{eq:ws}. If $\inf_m s_m/m>0$, then $d(\ww,q)$ has a unique unconditional basis.
\end{theorem}

Ari\~no et al.\@ \cite{AEP1988} investigated the Mackey topology of $d(\ww,q)$ in the case where $0<q<\infty$ and $\ww$ is nondecreasing. The memoir \cite{CRS2007} overrode all of these partial results on the Mackey topology by describing the associated space of every Lorentz sequence space (see \cite{CRS2007}*{Theorem 2.4.14}). The convenience of naming Lorentz sequence spaces after their fundamental functions arose within interpolation theory (see, e.g., \cite{CerdaColl2003}). The papers \cites{AABW2021,ABW2023} also contain valuable information regarding Lorentz sequence spaces.

\subsection{Regularity properties of weights}
We say that a nonnegative sequence $\prim=(s_m)_{m=1}^\infty$ is \emph{essentially increasing} (resp., \emph{essentially decreasing}) if $s_m\lesssim s_n$ (resp., $s_n\lesssim s_m$) for $m$, $n\in\NN$ with $m\le n$. Clearly, $\prim$ is essentially increasing (resp., decreasing) if and only if it is equivalent to a nondecreasing (resp, nonincreasing) sequence. If $s_m\lesssim s_n$ (resp., $s_n\lesssim s_m$) for $m$, $n\in\NN$ with $m\le n \le 2m$ we say that $\prim$ \emph{does not move down} (resp., \emph{up}) \emph{steeply}.

\begin{lemma}\label{lem:new}
Let $\prim=(s_m)_{m=1}^\infty$ be a sequence in $[0,\infty)$. Assume that there is $b\in\ZZ\cap[2,\infty)$ such that
$s_m \le s_{bm}$ (resp., $s_{bm}\le s_m$) for all $m\in\NN$.
\begin{enumerate}[label=(\alph*),leftmargin=*]
\item\label{it:new:A} If $\prim$ does not move up (resp., down) steeply, then $\prim$ does not move down (resp., up) steeply.

\item\label{it:new:B} If $\prim$ does not move down (resp., up) steeply, then $\prim$ is essentially increasing (resp., decreasing).
\end{enumerate}
\end{lemma}

\begin{proof}
We will prove the case when $s_m\le s_{bm}$ for all $m\in\NN$. Pick $j\in\NN$ such that $b\le 2^j$.

To see \ref{it:new:A}, we choose $C\in (0,\infty)$ such that $s_n\le C s_m$ as long as $m\le n \le 2m$. Given $m$, $n\in\NN$ with $m\le n\le 2m$, since $n\le bm \le 2^j n$,
\[
s_m \le s_{b m} \le C^j s_n.
\]

To see \ref{it:new:B}, we choose $D\in (0,\infty)$ such that $s_m\le D s_n$ as long as $m\le n \le 2m$. Given $m$, $n\in\NN$ with $m\le n$, there is $k\in\NN$ such that $b^{k-1}m\le n \le b^km$. Since $n\le 2^j b^{k-1}m$,
\[
s_m\le s_{b^{k-1}m}\le D^j s_n.\qedhere
\]
\end{proof}

The \emph{dual sequence} of the positive sequence $\prim=(s_m)_{m=1}^\infty$ is given by
\[
\prim^*=(m/s_m)_{m=1}^\infty.
\]
We say that $\prim$ \emph{lower regularity property} (LRP for short) if there is an integer $b\ge 2$ such
\begin{equation*}
2 s_m \le s_{bm}, \quad m\in\NN.
\end{equation*}
In turn, we say that $\prim=(s_m)_{m=1}^\infty$ has the \emph{upper regularity property} (URP for short) if
\begin{equation}\label{eq:URP}
s_{bm}\le \frac{b}{2} s_m, \quad m\in\NN,
\end{equation}
for some $b\in\NN$. Note that $\prim$ has the URP if and only if $\prim^*$ has the LRP.
\begin{lemma}\label{lem:URPPower}
Let $\prim=(s_m)_{m=1}^\infty$ be a doubling sequence in $(0,\infty)$. Then there is $\alpha\in(0,\infty)$ such that $\prim^\alpha$ has the URP.
\end{lemma}

\begin{proof}
Set $C=\sup_{m\in \NN} s_{2m}/s_m$. Pick $\alpha>0$ small enough so that $C^\alpha\le \sqrt{2}$. We have $s_{2m}^\alpha \le \sqrt{2} s_m^\alpha$ for all $m\in\NN$. Therefore, $s_{4m}^\alpha \le 2 s_m^\alpha$ for all $m\in\NN$.
\end{proof}

We define the \emph{multiplicative ratio} of $\prim$ as
\[
\rho\colon\NN\to(0,\infty], \quad d\mapsto\rho(d)=\sup_{m\in\NN} \frac{s_{dm}}{s_m}, \quad b\in\NN.
\]
The sequence $\rho$ is \emph{submultiplicative}, that is, $\rho(cd)\le \rho(c) \rho(d)$ for all $c$, $d\in\NN$. In this terminology, $\prim$ is doubling if and only if $\rho(2)<\infty$, and $\prim$ has the URP if and only if $\rho(b)\le b/2$ for some $b\in\NN$.

\begin{lemma}\label{lem:newdos}
Let $\prim=(s_m)_{m=1}^\infty$ be a sequence in $(0,\infty)$. If $\prim$ does not move down steeply, then its increment ratio $\rho$ does not either.
\end{lemma}

\begin{proof}
Set $C=\sup_{m\le n\le 2m} s_m/s_n$. Let $b\le d \le 2b$. Given $m\in\NN$,
\[
s_{bm}\le C s_{dm} \le \rho(d) s_m.
\]
Hence, $\rho(b) \le C \rho(d)$.
\end{proof}

Loosely speaking, an increasing sequence has the URP if and only if it increases steadily and its multiplicative ratio is smaller than that of $\primp_1$. Further, the URP is related to Dini's regularity.

\begin{lemma}\label{lem:URPImproved}
Let $\prim=(s_m)_{m=1}^\infty$ and $\primt$ be sequences in $(0,\infty)$ with $\primt\approx\prim$. Choose $\prim_\alpha=(s_m/m^\alpha)_{m=1}^\infty$ for all $0<\alpha\le 1$ and consider the sequence $\primv=(v_m)_{m=1}^\infty$ defined by
\[
v_m=\sum_{n=1}^m \frac{1}{s_n}.
\]
Let $\rho$ be the multiplicative ratio of $\prim$. Assuming that $\prim$ does not move down steeply, the following are equivalent.

\begin{enumerate}[label=(\alph*), leftmargin=*]
\item\label{eq:URP:A} $\prim$ has the URP.
\item\label{eq:URP:B} $\lim_d \rho(d)/d=0$.
\item\label{eq:URP:C} $\inf_d \rho(d)/d=0$.
\item\label{eq:URP:D} $\inf_d \rho(d)/d<1$.
\item There is $0<\alpha<1$ such that $\prim_\alpha$ is essentially decreasing.
\item $\primt$ has the URP.
\end{enumerate}
Moreover, if $\prim$ has the URP, then $\primv\lesssim\prim^*$ and $\prim^*$ is essentially increasing.
\end{lemma}

\begin{proof}
Set $\prim^*=(s_m^*)_{m=1}^\infty$ and $C=\sup_{m\le n\le 2m} s_m/s_n$.

Assume that $d\in\ZZ\cap[2,\infty)$ is such that $\rho(d)/d<1$. Since $\rho$ is submultiplicative, $\lim_k d^{-k}\rho(d^k)=0$. By Lemma~\ref{lem:newdos}, $\lim_b \rho(b)/b=0$. This proves that \ref{eq:URP:A}, \ref{eq:URP:B}, \ref{eq:URP:C} and \ref{eq:URP:D} are equivalent. In turn, these equivalences yield that $\prim$ has the URP if and only if $\primt$ has the URP.\\
Assume that $\prim$ has the URP, and let $b\in\NN$ be as in \eqref{eq:URP}. Pick $0<\alpha<1$ such that $b^{1-\alpha}\le 2$. Set $u_m=s_m/m^\alpha$ for all $m\in\NN$. We have $u_{bm} \le u_m$ for all $m\in\NN$. Since $\prim_\alpha$ inherits from $\prim$ the property of not moving down steeply, $\prim_\alpha$ is essentially decreasing by Lemma~\ref{lem:new}.

Assume now that $0<\alpha<1$ is such that $D:=\sup_{n\le m} u_m/u_n<\infty$. Then $\sup_{n\le m} s^*_n/s_m^*\le D$. For all $m\in\NN$ we have
\[
v_m
=\sum_{n=1}^m \frac{1}{n^\alpha u_n}
\le D \frac{1}{u_m} \sum_{n=1}^m \frac{1}{n^\alpha}
\le \frac{D}{1-\alpha} \frac{m^{1-\alpha}}{u_m}
=\frac{D}{1-\alpha} \frac{m}{s_m}
=\frac{D}{1-\alpha} s_m^*.
\]
Let $b\in\NN$ be such that $2D\le b^{1-\alpha}$. For all $m\in\NN$ we have
\[
s_{bm} \le D b^\alpha s_m \le \frac{b}{2} s_m.
\]
Hence, $\prim$ has the URP.
\end{proof}

\begin{lemma}\label{lem:DiniLRP}
Let $\prim=(s_m)_{m=1}^\infty$ be a sequence in $(0,\infty)$. Let $\primv=(v_m)_{m=1}^\infty$ be as in Lemma~\ref{lem:URPImproved}. If $\prim$ does not move up steeply and $\primv\lesssim \prim^*$, then $\primv$ has the LRP.
\end{lemma}

\begin{proof}
Set $C=\sup_{m\le n \le 2m} s_n/s_m<\infty$ and $D=\sup_m s_m\, v_m/m<\infty$. For all $m\in\NN$ we have
\[
\frac{v_{2m}-v_m}{v_m}\ge c:=\frac{1}{CD}.
\]
Therefore, $v_{2m} \ge (1+c) v_m$ and we infer that $\primv$ has the LRP.
\end{proof}

\begin{proposition}[cf.\@ \cite{AlbiacAnsorena2016}*{Lemma 2.12}]\label{prop:URPDini}
Assume that a sequence $\prim$ in $(0,\infty)$ is essentially increasing. Let $\primv$ as in Lemma~\ref{lem:URPImproved}. The following are equivalent.

\begin{enumerate}[label=(\alph*), leftmargin=*]
\item\label{eq:URP:AAA} $\prim$ has the URP.
\item\label{eq:URP:DDD} $\primv\lesssim\prim^*$.
\item\label{eq:URP:EEE} $\primv\approx\prim^*$.
\end{enumerate}
\end{proposition}

\begin{proof}
Lemma~\ref{lem:URPImproved} gives that \ref{eq:URP:AAA} implies \ref{eq:URP:DDD}. Our assumption gives $\prim^*\lesssim \primv$, whence \ref{eq:URP:DDD} implies \ref{eq:URP:EEE}. Assume that \ref{eq:URP:EEE} holds. Since $\primv$ is nondecreasing and $\primv^* \approx \prim$, $\prim$ does not move up steeply. By Lemma~\ref{lem:DiniLRP}, $\primv$ has the LRP. Therefore, $\primv^*$ has the URP. Consequently, $\prim$ has the URP as well.
\end{proof}
%----------------------------
\subsection{Boundedness of the averaging projections}
%--------------------------
Roughly speaking, the strategy to estimate the norm of the averages is as follows. Given an averaging projection, one splits the averaging blocks into two classes. The contribution of the unfavorable blocks is controlled through a comparison argument that exploits the symmetry of the canonical basis together with the regularity of the weight. The favorable blocks, on the other hand, can be estimated directly from the Lorentz quasi-norm by a rearrangement argument. This yields uniform boundedness of all averaging projections under lower and upper regularity assumptions on the weight.

Let us first introduce some additional terminology that we will need. The \emph{support} of $ f=(a_n)_{n=1}^\infty\in\FF^\NN$ will be the support of $f$ relative to the unit vector system, that is,
\[
\supp(f)=\enbrace{n\in\NN \colon a_n\not=0}.
\]
We denote by $\Dt$ the class of all families $\At=(A_j)_{j\in J}$ such that
\begin{itemize}
\item $A_j\subset \NN$ and $1\le \abs{A_j}<\infty$ for each $j\in J$;
\item $A_j\cap A_k=\emptyset$ whenever $j$, $k\in J$ and $j\not=k$.
\end{itemize}
Given such a family $\At$, we set $R(\At)=\bigcup_{j\in J} A_j$ and we denote by
\[
V_\At\colon \FF^\NN \to \FF^\NN
\]
the averaging projection relative to $\At$ and the unit vector system of $\FF^\NN$. If $\XX$ is an atomic quasi-Banach lattice and we regard $V_\At$ as an operator from $\XX$ into $\XX$, we set $V_\At=V_\At[\XX]$. The possibility that $J$ might be finite (or even empty) is not excluded. As we will use the convention that any sum over the empty set is zero, in the latter case, $V_\At$ is the zero operator

Given to families $\At=(A_j)_{j\in J}$ and $(B_j)_{j\in J}$ in $\Dt$ modeled over the same set $J$, we say that $\At\le\Bt$ if $A_j\subset B_j$ for all $j\in J$.

We record below some elementary facts about averaging projections on atomic quasi-Banach lattices that we will use below.

\begin{enumerate}[label=(A.\arabic*),leftmargin=*]
\item\label{it:positive} $V_\At$ is positive operator, that is, $\abs{V_\At(f)} \le V_\At(\abs{f})$ for all $f\in \FF^\NN$.
\item\label{it:perm} If $\Bt$ is a permutation of $\At$ then $V_\Bt=V_\At$.
\end{enumerate}

If a nondecreasing sequence $\prim$ in $(0,\infty)$ has the URP, then $\prim$ is doubling because
\[
s_{2m}\le s_{bm}\le \frac{b}{2} s_m,\qquad m\in \NN.
\]
Thus, $d_q(\prim)$ is a quasi-Banach space. The following result will allow us to restrict ourselves to spaces $d_q(\prim)$ where $\prim$ is a decreasing weight.

\begin{lemma}\label{lem:change}
Let $0<q\le 1$ and $\prim=(s_n)_{n=1}^\infty$ be a nondecreasing sequence with the URP and the LRP. Then there is a nonincreasing weight $\ww=(w_n)_{n=1}^\infty$ such that $d_q(\prim)=d(\ww,q)$.
\end{lemma}

\begin{proof}
By Lemma~\ref{lem:URPImproved}, $\prim^q$ has the URP and the LRP. Hence, also by Lemma~\ref{lem:URPImproved}, $(\prim^q)^*$ is essentially increasing and has the URP. Let $\ww=(w_m)_{m=1}^\infty$ be a nonincreasing weight equivalent to $(s_m^q/m)_{m=1}^\infty$. By Proposition~\ref{prop:URPDini},
\[
\sum_{n=1}^m w_n \approx \sum_{n=1}^m \frac{s_n^q}{n} \approx s_m^q, \quad m\in\NN.
\]
Therefore, $d_q(\prim)=d(\ww,q)$.
\end{proof}

Next we see a technical lemma whose proof relies on the rearrangement inequality.
\begin{lemma}\label{lem:key}
Let $0<q<\infty$, $\ww=(w_n)_{n=1}^\infty$ be a nonincreasing weight, $f=(a_n)_{n=1}^\infty\in c_{00}^+$ ,and $b$, $d$, $k\in\NN$. Let $\rho$ be the multiplicative ratio of $\primt=(t_m)_{m=1}^\infty$, where
\[
t_m=\sum_{j=1}^m w_n.
\]
Let $(m_j)_{j=0}^k$ be a nondecreasing $(k+1)$-tuple in $\ZZ$ with $m_0=0$. Let $J\subset\ZZ\cap[1,k]$, $(A_j)_{j\in J}\in \Dt$, and $(x_j)_{j\in J}$ be a nonincreasing sequence in $[0,\infty)$. Assume that
\[
b \abs{A_j} \le m_j-m_{j-1} \le d \abs{A_j}
\]
for all $j\in J$, and that $x_j \le a_n$ for all $j\in J$ and $n\in A_j$. Then
\begin{align*}
\sum_{j\in J} x_j^q \enpar{t_{m_j} - t_{m_{j-1}}} \le \frac{d \rho(b)}{b} \norm{f}_{(\ww,q)}^q.
\end{align*}
\end{lemma}

\begin{proof}
The primitive sequence $\primt$ satisfies the convexity inequality
\begin{equation}\label{eq:convex}
\frac{t_{l_1}-t_{l_2}}{l_1-l_2} \le \frac{t_{l_3}-t_{l_4}}{l_3-l_4}, \quad l_3\le l_1<l_2, \, l_3<l_4\le l_2.
\end{equation}
Put $t_0=0$, $r_0=0$, $r=\sum_{j\in J} \abs{A_j}$, and
\[
r_j=\sum_{\substack{i\in J\\ i \le j}} \abs{A_i}, \quad j\in J.
\]
Set $\pi(j)=0$ if $j=\min(J)$. For any other $j\in J$ let $\pi(j)$ be its predecessor in $J$.
Pick a bijection
\[
\psi\colon \ZZ\cap [1,r] \to \bigcup_{j\in J} A_j
\]
such that $\psi(n)\in A_j$ if $r_{\pi(j)}<n\le r_j$. By the rearrangement inequality,
\[
\norm{f}_{(\ww,q)}^q \ge \sum_{n=1}^r a_{\psi(n)}^q w_n
\ge
\sum_{j\in J} x_j^q \enpar{t_{r_j}-t_{r_{\pi(j)}}}.
\]
Since $(m_j)_{j=0}^k$ is nondecreasing, for any $j\in J$ we have
\[
b\, r_j
\le\sum_{\substack{i\in J\\ i \le j}} m_i-m_{i-1}
\le \sum_{i=1}^{j} m_i-m_{i-1}
=m_j.
\]
Consequently, $b\, r_{\pi(j)} \le m_{\pi(j)} \le m_{j-1}$ for all $j\in J$. By \eqref{eq:convex},
\begin{multline*}
\sum_{\substack{i\in J\\ i \le j}} t_{r_i}-t_{r_{\pi(i)}}
=t_{r_j}\ge \frac{1}{\rho(b)} t_{br_j}
=\frac{1}{\rho(b)} \sum_{\substack{i\in J\\ i \le j}} t_{b r_i}-t_{br_{\pi(i)}}\\
\ge \frac{1}{\rho(b)} \sum_{\substack{i\in J\\ i \le j}} \frac{b \abs{A_i}}{m_i-m_{i-1}} \enpar{ t_{m_i}-t_{m_{i-1}}}
\ge \frac{b}{d \rho(b)} \sum_{\substack{i\in J\\ i \le j}} t_{m_i}-t_{m_{i-1}}.
\end{multline*}
Applying Abel's summation formula finishes the proof.
\end{proof}

The following lemma will allow us to choose a convenient element from $\Dt$ in the proof of Theorem~\ref{thm:general}. Notice that we assume the space under consideration to be $q$-Banach for some $0<q\le 1$ only to make sure that its quasi-norm is continuous.

\begin{lemma}\label{lemma: max}
Suppose that $\XX$ is a rearrangement invariant atomic quasi-Banach lattice. Assume that $\XX$ is equipped with a $q$-norm for some $0<q\le 1$, and
\[
\lim_{m\in\NN} \frac{\Gamma_m[\XX]}{m}=0.
\]
Let $f\in\XX$ be finitely suported. Then the set
\begin{equation}\label{eq:AveNorm}
\enbrace{\norm{V_\At(f)} \colon \At\in\Dt}
\end{equation}
attains its supremum.
\end{lemma}

\begin{proof}
Consider the set
\[
\Dt[f]=\enbrace{\Et \in\Dt \colon R(\Et)=\supp(f)}.
\]

Given $\At\in\Dt$, let $\Bt\in\Dt$ be the family obtained from $\At$ by dropping all entries that do not overlap $\supp(f)$ and, if $\supp(f)\not\subset R(\At)$, appending the entry $\supp(f)\setminus R(\At)$ to $\At$. Since $\abs{V_\At(f)}\le \abs{V_\Bt(f)}$, we have $\norm{V_\At(f)}\le \norm{V_\Bt(f)}$. Moreover, $\Bt\ge \Et$ for some $\Et\in\Dt[f]$. Hence, the supremum value of the set \eqref{eq:AveNorm} is approached through its subset
\[
\enbrace{\norm{V_{\Bt}(f)} \colon \Bt\in\Dt, \, \Et\le \Bt \mbox{ for some } \Et\in\Dt[f]}.
\]

Fix $\Et=(E_j)_{j\in J} \in \Dt[f]$. For each $j\in J$, let $a_j$ be the value of $V_\Et(f)$ on $E_j$. Given $\Bt=(B_j)_{j\in J}\in\Dt$ with $\Bt\ge\Et$,
\[
V_\Bt(f)=\sum_{j\in J} a_j \abs{E_j} \frac{\sum_{n\in B_j} \ee_n}{\abs{B_j}}.
\]
Since $J$ is finite and, by assumption,
\[
\lim_{\abs{B_j}\to \infty} \frac{\sum_{n\in B_j} \ee_n}{\abs{B_j}}=0,
\]
there is $\Bt(\Et)\in\Dt$ with $\Et\le \Bt(\Et)$ such that the supremum value of
\[
\enbrace{\norm{V_{\Bt}(f)} \colon \Bt\in\Dt ,\, \Bt\ge \Et}
\]
is approached through its subset
\[
\enbrace{\norm{V_{\Bt}(f)} \colon \Bt\in\Dt, \, \Et\le\Bt \le\Bt(\Et)}.
\]

Summing up, the supremum value of the set \eqref{eq:AveNorm} is approached through its subset
\[
\enbrace{\norm{V_{\Bt}(f)} \colon \Bt\in\Dt, \, \Et\le \Bt\le \Bt(\Et) \mbox{ for some } \Et\in\Dt[f]}.
\]
Since this set is finite, the proof is over.
\end{proof}

We are now ready for the main result of this section. Note that given a nondecreasing unbounded sequence $\prim$ and $q\in(0,\infty)$, the space $d_q(\prim)$ is a Banach space if and only if $q>1$ and $\prim$ has the URP, or $q=1$ and $\prim^*$ is essentially increasing (see \cite{ABW2023}*{Theorem 3.8}). So, Theorem~\ref{thm:general} answers Question~\ref{qt:LNCAP} in the affirmative.
\begin{theorem}\label{thm:general}
Let $0<q<1$ and $\prim$ be a nondecreasing sequence with the URP and the LRP. Then, there is a constant $C$ such that
\[
\norm{V_\At[d_q(\prim)]}\le C
\]
for all $\At\in\Dt$.
\end{theorem}

\begin{proof}
Use Lemma~\ref{lem:change} to pick a nonincreasing weight $\ww=(w_n)_{n=1}^\infty$ such that $d_q(\prim)=d(\ww,q)$. Let $\rho$ be the multiplicative ratio of the sequence $\primt=(t_m)_{m=1}^\infty$ defined by
\[
t_m=\sum_{n=1}^m w_n.
\]
Since $\primt\approx \prim^q$, $\lim_d \rho(d) d^{-q}=0$ by Lemma~\ref{lem:URPImproved}. Let $b\in\NN$ be such that
\[
\rho(b)\le 2^{-2-q} (1+b)^q.
\]
Pick $f\in c_{00}$. Set $\abs{f}=g=(a_n)_{n=1}^\infty$ and $k:=|\supp(f)|$. Assume that $f\not=0$, so that $k\in \NN$. Since $\prim$ has the URP, $\lim_m t_m/m=0$ by Lemma~\ref{lem:URPImproved}. Hence, we can apply Lemma~\ref{lemma: max} to obtain $\Bt=(B_j)_{j=1}^k \in\Dt$ such that
\begin{equation*}
\norm{V_\At(g)}_{(\ww,q)} \le \norm{V_\Bt(g)}_{(\ww,q)}, \quad \At\in\Dt.
\end{equation*}
Further, by \ref{it:perm}, we may choose $\Bt$ so that the $k$-tuple $(x_j)_{j=1}^k$ defined by
\[
x_j =\frac{1}{\abs{B_j}} \sum_{n\in B_j} a_n
\]
is nonincreasing. Put $m_j=\sum_{i=1}^j \abs{B_i}$ for all $j\in \ZZ\cap[0,k]$. For each $j\in\ZZ\cap[1,k]$ put
\[
D_j=\enbrace{n \in B_j \colon a_n \ge \frac{x_j}{2}}.
\]
Now define
\[
J=\enbrace{j\in\NN\cap[1,k] \colon \abs{D_j}< \frac{\abs{B_j}}{b}},
\]
and $G=\NN\cap[1,k]\setminus J$. For each $j\in J$ choose $D_j \subset E_j \subset B_j$ such that
\[
\abs{E_j}=\enfloor{ \frac{\abs{B_j}}{b} }.
\]
Since $D_j\not=\emptyset$ for all $j\in\ZZ\cap[1,k]$ it follows that $\Et:=(E_j)_{j\in J}\in \Dt$. Set
\[
y_j=\frac{1}{\abs{E_j}} \sum_{n\in E_j} a_n, \quad j\in J.
\]
We have
\[b\abs{E_j} \le \abs{B_j}\]
and
\[
\abs{B_j} x_j =\abs{E_j} y_j+ \sum_{n\in B_j \setminus E_j}a_n
\le \abs{E_j} y_j + \abs{B_j \setminus E_j} \frac{x_j}{2}.
\]
Consequently,
\[
x_j \le \frac{2\abs{E_j}}{2\abs{B_j}-\abs{B_j\setminus E_j}}y_j= \frac{2\abs{E_j}}{\abs{B_j}+\abs{E_j}}y_j
\le \frac{2}{1+b} y_j.
\]
For $j\in \ZZ\cap[0,k]$ put $m_j=\sum_{i=1}^j \abs{B_i}$ and define
\[
R=\sum_{j\in J} x_j^q \enpar{t_{m_j} - t_{m_{j-1}}}, \quad S=\sum_{j\in G} x_j^q \enpar{t_{m_j} - t_{m_{j-1}}},
\]
so that $\norm{V_\Bt(g)}_{(\ww,q)}= (R+S)^{1/q}$. Notice that
\[b\abs{E_j}\le \abs{B_j}\le b(\abs{E_j}+1)\le 2b\abs{E_j},\quad j\in J.
\]
By Lemma~\ref{lem:key},
\begin{align*}
R\le& \frac{2b\rho(b)}{b}\frac{2^q}{(1+b)^q}\norm{V_\Et(g)}_{(\ww,q)}^q \le \frac{1}{2} \norm{V_\Bt(g)}_{(\ww,q)}^q=\frac{R+S}{2}.
\end{align*}
Therefore, $R\le S$, whence $\norm{V_\Bt(g)}_{(\ww,q)} \le 2^{1/q} S^{1/q}$. Since
\[
\abs{D_j} \le m_j-m_{j-1} \le b \abs{D_j}, \quad j\in G,
\]
and $\rho(1)=1$, another application of Lemma~\ref{lem:key} yields
\[
S\le b 2^q \norm{g}_{(\ww,q)}^q = b 2^q \norm{f}_{(\ww,q)}^q.
\]
Pick $\At\in\Dt$. By \ref{it:positive},
\[
\norm{V_\At(f)}_{(\ww,q)}
\le \norm{V_\At(g)}_{(\ww,q)}
\le\norm{V_\Bt(g)}_{(\ww,q)}
\le 2^{1+1/q} b^{1/q} \norm{f}_{(\ww,q)}.\qedhere
\]
\end{proof}

The sequence $\primp_p=(m^{1/p})_{m=1}^\infty$ has the URP and the LRP when $1<p<\infty$. Thus, Theorem~\ref{thm:general} applies in particular to classical Lorentz spaces $\ell_{p,q}$ for $0<q<1<p<\infty$. It also provides an alternative solution to the problem raised by Popa in \cite{Popa1981} of whether $\ell_q$ is a complemented subspace of the Lorentz sequence space $d(\ww,q)$, where $0<q<1$ and $\ww$ is a nonincreasing weight. In a sense, our approach here is more natural than Nawrocki and Orty\'nski's, since it does not require to look for strenuous projections but uses instead canonical projections, once they are guaranteed to be bounded.

\begin{corollary}[cf.\@ Theorem~\ref{thm:NO}]\label{cor:lqCom}
Let $0<q<1$ and $\prim$ be a nondecreasing sequence with both the URP and the LRP. Then $d_q(\prim)$ has a constant-coefficients block basic sequence that is complemented and equivalent to the canonical $\ell_q$-basis.
\end{corollary}

\begin{proof}
Just combine Theorem~\ref{thm:general}, Lemma~\ref{lem:change} and Theorem~\ref{thm:Popa}.
\end{proof}
% ------------------------------------------------------------------------
\section{Existence of conditional almost greedy bases in Lorentz sequence spaces}\label{sect:TGA}\noindent
% ------------------------------------------------------------------------
The Dilworth--Kalton--Kutzarova (DKK for short) was originally designed to show the existence of almost greedy conditional bases in certain types of Banach spaces. It explicitly shows how to construct such bases from three ingredients. Namely,
\begin{itemize}[leftmargin=*]
\item A semi-normalized Schauder basis $\XB=(\xx_k)_{k=1}^\infty$ of a Banach space $\XX$,
\item A symmetric or, more generally, subsymmetric Schauder basis $\St=(\st_n)_{n=1}^\infty$ of another Banach space $\Sym$, and
\item a partition of $\Jt=(J_k)_{k=1}^\infty$ of $\NN$ with $J_k\not=\emptyset$ and $\max(J_k)<\min(J_{k+1})$ for all $n\in\NN$.
\end{itemize}

Specifically, the Banach space $\YY[\XB,\St,\Jt]$ constructed from the DKK method is the completion of $c_{00}$ with respect to the norm
\begin{multline*}
f=(a_n)_{n=1}^\infty \mapsto \norm{f}_{\XB\St,\Jt}=
\norm{ \sum_{k=1}^\infty \frac{ 1 }{\abs{J_k}} \enpar{ \sum_{n\in J_k} a_n } \norm{\sum_{j\in J_k} \st_j} \xx_k }\\
+\norm{\sum_{n=1}^\infty a_n\, \st_n -
\sum_{k=1}^\infty \frac{1}{\abs{J_k}} \enpar{\sum_{n\in J_k} a_n } \enpar{\sum_{j\in J_k} \st_j }}.
\end{multline*}

Although the DKK method was originally designed to show the existence of almost greedy conditional bases in certain types of Banach spaces, the authors of \cite{AABW2024} adapted the method to prove similar structural results in a more general framework. They considered the case where $\XX$ is a quasi-Banach space and kept the assumption that $\Sym$ be locally convex. Their reason for not trying to make the extension work in case that $\Sym$ lacked local convexity was precisely the fact that the DKK method depends on the boundedness of the averaging projections associated with symmetric (or subsymmetric) bases and increasing partitions, and they did not expect a positive answer to Question~\ref{qt:LNCAP}! In hindsight, \cite{AABW2024} could have been written verbatim by feeding the DKK method with a quasi-Banach space $\Sym$ with a subsymmetric basis $\St$ such that the averaging projections associated with $\St$ and $\Jt$ are bounded. We leave the tedious details for the reader and record a result from \cite{AABW2024} adapted in this direction.

\begin{theorem}[cf.\@ \cite{AABW2024}*{Theorem 4.1}]\label{thm:AABW}
Let $\XX$ be a quasi-Banach space with a Schauder basis $\XB=(\xx_n)_{n=1}^\infty$, and let $\Sym$ be a quasi-Banach space with a subsymmetric basis $\St$. Assume that the averaging projections relative to $\St$ are bounded. Let $\delta\colon[0,\infty) \to [0,\infty)$ be a nondecreasing map such that
\[
\uncr_m[\XB] \gtrsim \delta(m),\quad m\in\NN.
\]
Suppose that $\XX$ is locally $p$-convex, $0<p\le 1$, and that $\St$ is $L$-concave. Let $1 <q\le s\le\infty$ be such that $\St$ satisfies an upper $q$-estimate and a lower $s$-estimate. Let $\varphi\colon[0,\infty) \to [0,\infty)$ be a concave increasing function. Then the space $\XX\oplus\Sym$ has an almost greedy Schauder basis $\YB$ with the following additional properties:
\begin{enumerate}[label=(\roman*),leftmargin=*,widest=iii]
\item $\Gamma_m[\YB]\approx \Gamma_m[\St]$ for $m\in\NN$.
\item $\uncr_{m}[\YB]\gtrsim \delta(\log m)$ for $m\ge 2$.
\end{enumerate}
Further, in the case when $\unc_m[\XB] \approx \uncr_m[\XB]\gtrsim m^{\max\{1/q,1/p-1/s\}}$ for $m\in\NN$ we can choose $\YB$ satisfying
\[
\unc_{m}[\YB] \approx \uncr_{m}[\YB]\approx \delta(\varphi(\log m)), \quad m\in\NN.
\]
\end{theorem}

With an eye to applying Theorem~\ref{thm:AABW} in the case when $\Sym$ is a nonlocally convex Lorentz sequence space, we state some results on their lattice structure.

\begin{theorem}[cf.\@ \cite{AlbiacAnsorena2022c}*{Theorem 7.1}]\label{thm:SLConvex}
Let $0<q,r<\infty$ and $\prim$ be a nondecreasing doubling unbounded sequence in $(0,\infty)$.
\begin{enumerate}[label=(\alph*), leftmargin=*]
\item\label{it:LC:a} $d_q(\prim)$ is lattice $r$-convex if and only if $r=q$ and $(\prim^r)^*$ is essentially increasing, or $r<q$ and $\prim^r$ has the URP.
\item\label{it:LC:b} $d_q(\prim)$ is L-convex.
\end{enumerate}
\end{theorem}

\begin{proof}
By \eqref{eq:LatticeConvex}, to prove \ref{it:LC:a} it suffices to consider the case where $r=1$. Since lattice $1$-convexity is just local convexity, the result in this particular case follows from \cite{ABW2023}*{Theorem 3.8}. In turn, \ref{it:LC:b} follows from \ref{it:LC:a} and Lemma~\ref{lem:URPPower}.
\end{proof}

As far as we know, lattice concavity of Lorentz sequence spaces is not so well understood. Nevertheless, Kami\'nska--Maligranda's work \cite{KamMal2004} on function Lorentz spaces provides valuable information on the discrete case.

\begin{theorem}
Let $0<q,r<\infty$ and $\prim$ be a nondecreasing doubling sequence in $(0,\infty)$.
\begin{enumerate}[label=(\alph*), leftmargin=*]
\item\label{it:LV:a} If $d_q(\prim)$ is lattice $r$-concave, then $(\prim^r)^*$ is essentially decreasing.
\item\label{it:LV:b} If $r>q$ and there is $s\in(q,r)$ such that $(\prim^s)^*$ is essentially decreasing, then $d_q(\prim)$ is lattice $r$-concave.
\item $d_q(\prim)$ is L-concave if and only if $\prim$ has the LRP.
\end{enumerate}
\end{theorem}

\begin{proof}
By \eqref{eq:FFL}, there are $C_1$, $C_2\in(0,\infty)$ such that
\[
\frac{1}{C_1} s_{\abs{A}} \le \norm{\sum_{n\in A} \ee_n}\le C_2 s_{\abs{A}}, \quad A\subset\NN, \, \abs{A}<\infty.
\]
Assume that $d_q(\prim)$ is lattice $r$-concave, and let $C$ be the associated concavity constant. Pick $m$, $n\in\NN$ with $m\le n$. Choose $A\subset\NN$ with $\abs{A}=m$. Let $\Pt_m(A)$ be the set consisting of all subsets of $A$ of cardinality $m$. Note that $\abs{\Pt_m(A)}=\binom{n}{m}$. The identity
\[
\sum_{B\in\Pt_m(A)} \sum_{n\in B}\ee_n= \binom{n-1}{m-1} \sum_{n\in A} \ee_n
\]
yields
\[
\frac{1}{C^r C_1^r} \binom{n}{m} s_m^r \le C_2^r \binom{n-1}{m-1} s_n^r.
\]
Since $\binom{n}{m}= \binom{n-1}{m-1} n/m$, \ref{it:LV:a} holds.

Assume $(\prim^s)^*$ is essentially decreasing, where $q<s<r$. Define $w\colon(0, \infty) \to (0,\infty)$ by
\[
w=\sum_{n=1}^\infty \enpar{s_n^q-s_{n-1}^q} \chi_{(n-1,n]}
\]
Let $W$ be the primitive of $w$, that is $W(t)=\int_0^t w(s)\, ds$ for all $t\in(0,\infty)$. Then, the function $V\colon(0,\infty)\to (0,\infty)$ defined by $V(t)=W^{s/q}(t)/t$ is essentially increasing. By \cite{KamMal2004}*{Theorem 6}, the function Lorentz space $\Lambda(w,q)$ is lattice $r$-concave. Since $d_q(\prim)$ is isomorphic to the quasi-Banach lattice consisting of all functions in $\Lambda(w,q)$ that are constant on each interval $(n-1,n]$, $n\in\NN$, \ref{it:LV:b} holds.

Our assumptions imply that $\prim^*$ does not moves down steeply. Assume that $\prim$ has the LRP. Then, $\prim^*$ has the URP. By Lemma~\ref{lem:URPImproved}, there is $0<\alpha_0<1$ such that $(m^{1-\alpha}/s_m)_{m=1}^\infty$ is essentially decreasing for all $\alpha_0<\alpha<1$. By \ref{it:LV:b}, $d_q(\prim)$ is lattice $r$-concave provided that $r>\max\{q,1/(1-\alpha_0)\}$.

Assume that $d_q(\prim)$ is lattice $r$-concave, $1<r<\infty$. By \ref{it:LV:a} and Lemma~\ref{lem:URPImproved}, $\prim^*$ has the URP. Hence, $\prim$ has the LRP.
\end{proof}

We close with an existence result that yields a continuum of mutually non-equivalent conditional (almost greedy) Schauder bases in a broad class of nonlocally convex Lorentz spaces. It applies, in particular, to Lorentz spaces $\ell_{p,q}$ for $0<q<1<p<\infty$.

\begin{theorem}\label{thm:CAGLorentz}
Let $0<q<1$, $\prim$ be a nondecreasing sequence with the URP and the LRP, $\XX$ be a quasi-Banach space with a Schauder basis $\XB$, and $\varphi\colon[0,\infty) \to [0,\infty)$ be a concave increasing function. Assume that $d_q(\prim)$ is isomorphic to a complemented subspace of $\XX$. Then $\XX$ has an almost greedy basis $\YB$ with $\unc_m[\YB]\approx \varphi^{1/q}(\log(m))$ for $m\ge 2$, and $(\Gamma_m[\YB])_{m=1}^\infty \approx \prim$.
\end{theorem}

\begin{proof}
Let $\Dt$ be the difference basis of $\ell_q$. It is known \cite{AABW2024}*{Propositon 3.2} that $\unc_m[\Dt]\approx \uncr_m[\Dt]\approx m^{1/q}$ for $m\in\NN$. Therefore, $\unc_m[\XB\oplus\Dt]\approx \uncr_m[\XB\oplus\Dt]\approx m^{1/q}$ for $m\in\NN$. In turn, $d_q(\prim)$ satisfies an upper $q$-estimate by Theorem~\ref{thm:SLConvex}. By Theorem~\ref{thm:general}, applying Theorem~\ref{thm:AABW} with $\XB\oplus\Dt$ and the unit vector system of $d_q(\prim)$ yields a Schauder $\YB$ of $\YY:=\XX\oplus\ell_q\oplus d_q(\prim)$ satisfying the whised-for properties. Since, by Corollary~\ref{cor:lqCom}, $\YY$ is isomorphic to $\XX$, the proof is over.
\end{proof}

\bigskip
% ------------------------------------------------------------------------
%\section*{Statements and Declarations}\noindent
% ------------------------------------------------------------------------
\subsection*{Conflict of interest}
The authors have no competing interests to declare that are relevant to the content of this article.

\subsection*{Data Availability}
Data sharing does not apply to this article as no datasets were generated or analysed during the current study.
% ------------------------------------------------------------------------
%\bibliography{Biblio}

%--------------------------------------
\end{document}